\documentclass{amsart}
\usepackage{amsfonts,amssymb,amscd,amsmath,enumerate,verbatim,calc}
\usepackage[all]{xy}
\SelectTips{eu}{}

\newcommand{\mf}{\mathfrak }

\newcommand{\xra}{\xrightarrow}
\newcommand{\arr}{\rightarrow}

\newcommand{\sr}{{{}^{\varphi}\!}}
\newcommand{\srh}{{{}^{\widehat \varphi}\!}}

\newcommand{\wh}{\widehat }

\newcommand{\rank}{\operatorname{rank}}
\newcommand{\type}{\operatorname{type}}
\newcommand{\Ann}{\operatorname{Ann}}
\newcommand{\Ass}{\operatorname{Ass}}

\newcommand{\Ker}{\operatorname{Ker}}

\newcommand{\Homology}{\operatorname{H}}
\newcommand{\Spec}{\operatorname{Spec}}

\newcommand{\Rhom}[3]{{\bf{R}\!\Hom}_{#1}(#2,#3)}
\newcommand{\dtensor}[3]{{#1}\otimes_{#2}^{\bf{L}}{#3}}

\newcommand{\depth}{\operatorname{depth}}

\newcommand{\Hom}{\operatorname{Hom}}

\newcommand{\Ext}{\operatorname{Ext}}
\newcommand{\ext}[4]{\operatorname{Ext}^{#1}_{#2}(#3,#4)}
\newcommand{\gam}[1][\mf m]{\Gamma_{\!#1}}
\newcommand{\lch}[1]{\operatorname{H}_{\!\mf m}^{#1}}

\theoremstyle{plain}

\theoremstyle{definition}

\theoremstyle{plain}

\newtheorem{theorem}{Theorem}[section]

\newtheorem{lemma}[theorem]{Lemma}

\newtheorem{corollary}[theorem]{Corollary}

\newtheorem{itheorem}{Theorem}

\theoremstyle{definition}

\newenvironment{mainthmproof}{%

\begin{proof}}{\end{proof}}

\newtheorem{chunk}[theorem]{}

\theoremstyle{remark}

\newtheorem{remark}[theorem]{Remark}

\newtheorem*{Remark}{Remark}

\numberwithin{equation}{theorem}

\input xy
\xyoption{all}

\begin{document}

\title[Contracting Endomorphisms ]
{Contracting Endomorphisms  and  \\ Gorenstein Modules}

\author{Hamid Rahmati}
\address{Department of Mathematics, University of Nebraska,  Lincoln, NE 68588, U.S.A.}
\email{hrahmati@math.unl.edu}

\thanks{Research partly supported through NSF grant DMS 0201904}

\keywords{Frobenius endomorphism, Gorenstein ring, homological dimension}

\subjclass[2000]{Primary 13C13, 13H10. Secondary 13D05}

\begin{abstract}
A finite module $M$ over a noetherian local ring $R$ is said to be
Gorenstein if $\Ext^i(k,M)=0$ for all $i \ne \dim R$. A endomorphism
$\varphi\colon R\to R$ of rings is called contracting if  $\varphi
^i(\mf m) \subseteq \mf m^2$ for some $i \geq 1$. Letting $\sr R $
denote the $R$-module $R$ with action induced by $\varphi$, we prove:
A finite $R$-module M is Gorenstein if and only if  $\rm {Hom}_R(\sr R,M)
\cong M$ and $\rm {Ext}_R^i(\sr R,M) = 0$ for $1 \leq i \leq \rm {depth}
R$. 
 \end{abstract}

\maketitle

\section*{Introduction}

Let  $(R, \mf m)$ be a local noetherian ring and $\varphi\colon R \to R$
a homomorphism of rings which is local, that is to say, $\varphi(\mf
m)\subseteq \mf m$.  Such an endomorphism $\varphi$ is said to be
\emph{contracting} if $\varphi^i(\mf m) \subseteq \mf m^2$ for some
non-negative integer $i$.

The Frobenius endomorphism of a ring of positive  characteristic is the
prototypical example of a contracting endomorphism.  It has been used
to study rings of prime characteristic and modules over those rings.
There are, however, other natural classes of contracting endomorphisms;
see \cite{AIM}, where Avramov, Iyengar, and Miller show that they enjoy
many properties of the Frobenius map. In this paper we provide two
further results in this direction.

In what follows, we write $\sr R$ for $R$ viewed as a module over
itself via $\varphi$; thus, $r\cdot x = \varphi(r)x$ for $r \in R$ and
$x \in \sr R$. The endomorphism $\varphi$ is \emph{finite} if $\sr R$
is a finite $R$-module. For any $R$-module $M$, we view $\Hom_R(\sr R,M)$
as an $R$-module with action defined by $(r\cdot f)(s)=f(sr)$ for $r\in R$,
$s\in \sr R$, and $f\in \Hom_{R}(R,M)$.

A finitely generated $R$-module $M$ is said to be \emph{Gorenstein}
if it is maximial Cohen-Macaulay and  has finite injective dimension.

\begin{itheorem}
\label{theorem1}
Let $\varphi \colon R \to R$ be a finite contracting endomorphism and $M$ a finite
$R$-module. The following are equivalent:
\begin{enumerate}[\quad \rm (a)]
\item  $R$ is Cohen-Macaulay and $M$ is Gorenstein.
\item $M\cong \Hom_R(\sr R,M)$  and one has $\Ext^i_R(\sr R,M) = 0$ for $i \geq 1$.
\item $M$ is a direct summand of $\Hom_R(\sr R,M)$ as an $R$-module, and one has
  $\Ext^i_R(\sr R,M) = 0$ for $1 \leq i \leq \depth M$.
\end{enumerate}
\end{itheorem}

This theorem generalizes a theorem of Goto~\cite[(1.1)]{G} for the Frobenius endomorphism;
see Corollary~\ref{maincorollary}. As another application, we give a short proof of a recent
result of Iyengar and Sather-Wagsaff \cite{IS} in the special case when $\varphi$ is
finite.

\begin{itheorem}
\label{theorem3} 
Let  $\varphi \colon R \to R$ be a finite contracting homomorphism. The
ring $R$ is Gorenstein if and only if $\operatorname{G-dim}_R(\sr R)$
is finite.  \end{itheorem}

Here $\operatorname{G-dim}$ denotes Gorenstein dimension, a notion
recalled in Section~\ref{sec:two}.

\section{Injective envelopes and change of rings}

In this section we prove a result that tracks injective modules under
finite change of rings.  We write $E_S(M)$ for the injective envelope
of an $S$-module $M$.

\begin{theorem}
\label{injbasechange} 
Let $\psi \colon S \to T$ be a finite homomorphism of noetherian rings. 

For each prime $\mf p $ in $S$ there is  an isomorphism of $T$-modules:
\[
\Hom_S(T,E_S(S/{\mf p})) \cong \bigoplus_{\underset  {\mf q \cap S  = \mf p
 }{\mf q \in \Spec T}}E_T(T/\mf q)
\]
\end{theorem}

\begin{proof}
 The $T$-module  $\Hom_S(T,E_S(S/{\mf p}))$ is  injective, since one has 
 \begin{align*}
 \Hom_T(-,\Hom_S(T,E_S(S/{\mf p}))) & \cong \Hom_T(- \otimes_T T,E_S(S/{\mf p})) \\ & \cong \Hom_S(-,E_S(S/{\mf p}))
 \end{align*}

Choose $\mf q$ in $\Ass_T \Hom_S(T,E_S(S/{\mf p}))$. The $S$-module $T$ is finite,
so one  has 
\[
\mf q \cap S \in  \Ass_ S \Hom_S(T,E_S(S/{\mf p})) = 
\operatorname {Supp}_S T \cap \Ass_S E_S(S/{\mf p}) = \lbrace \mf p \rbrace
\]
We have proved  $\mf q \cap S =\mf p$. Set $A= \lbrace \mf q \in \Spec T \
\vert \ \mf q \cap S = \mf p \rbrace$. The structure theorem for injective
modules over noetherian rings yields numbers $i_{\mf q}$ such that
\[
\Hom_S(T,E_S(S/{\mf p})) \cong \bigoplus_{\mf q \in A} E_T(T/{\mf q})^{i_{\mf q}}
\]

It remains to show that $i_{\mf q}= 1$ for all $\mf q \in A$.
Setting $k(\mf q) = T_{\mf q}/{\mf q T_{\mf q}}$, one has 
\[
i_{\mf q} = \rank_{ k(\mf q)} \Hom_{T_{\mf q}}\big({k(\mf q)}, \Hom_S(T,E_S(S/{\mf p}))_{\mf q}\big)
\]

The isomorphism $E_S(S/{\mf p}) \cong E_{S_{\mf p}}(S_{\mf p}/{\mf p
S_{\mf p}})$ yields the first isomorphism below:
\begin{align*}
\Hom_S(T,E_S(S/{\mf p}))_{\mf q} & \cong \Hom_S(T,E_{S_{\mf p}}(S_{\mf p}/{\mf p S_{\mf p}}))_{\mf q} \\ & \cong \Hom_{S_{\mf p}}(T_{\mf p},E_{S_{\mf p}}(S_{\mf p}/{\mf p S_{\mf p}}))_{\mf q}
\end{align*}
The second isomoprhism is due to finitness of $T$  over $S$. Thus, we
may assume that $S$ is local with maximal ideal $\mf p$, and $\mf q $
is  maximal in $T$.

Set $k=S/{\mf p}$ and $l=T/\mf q$.  One then has isomorphisms of
$l$-vector spaces
 \begin{align*}
\Hom_T(l,\Hom_S(T,E_S(k)) & \cong \Hom_S(l,E_S(k))
\\ & \cong \Hom_S(l,k)
\\ & \cong\Hom_k(l,k))
 \end{align*}
Here the first isomorphism holds by adjointness, the second one because
maps from $l$ to $E_S(k)$ factor through $k$, the last one because
$S$-linear maps from $l$ amnnihilate $\mf q$.  Since $\rank_k l$ is
finite, we get $\rank_k(\Hom_T(l,\Hom_S(T,E_S(k)))=1$.
 \end{proof}

\begin{remark}
Every injective $S$-module $I$ is a direct sum of modules of the form
$E(S/\mf p)$, and the functor $\Hom_{S}(T,-)$ commutes with direct sums,
so Theorem~\ref{injbasechange} provides
a direct sum decomposition of the injective $T$-module $\Hom_{S}(T,I)$.
This result may be compared to Foxby's computation of the injective
dimension of $T \otimes_S I$, when $\psi\colon S\to T$ is a flat local
homomorphism; see \cite[Theorem 1]{Fo}.
 \end{remark}

\begin{remark}
Theorem~\ref{injbasechange} may fail when $\psi$ is
not finite, even if it is flat. Indeed, let $S$ be a field and $T$ an
infinite field extension of $S$ with a countable basis. One then has $E_S(S)
= S$, and hence the $S$-vector space  $\Hom_S(T,E_S(S)) $ does not have
a countable basis. Thus, it cannot be isomorphic to the $S$-vector space
$E_T(T) =T$.  \end{remark}

\section{Gorenstein Modules}
\label{sec:two}

In this section we prove Theorem~\ref{theorem1}, by way of Lemmas~\ref{freeinj} and
\ref{localcohomology} below. 

Given an $R$-module $M$, we view $M\otimes_{R} \sr R$ as an $R$-module with $r\cdot
(x\otimes s) = x \otimes sr$.  Recall that $\Hom_R(\sr R,N)$ is to be viewed as an
$R$-module with $(r\cdot f)(s) = f(sr)$.

J.~Herzog~\cite[(4.3), (5.1)]{H} has proved the following result in the special case when
$\varphi$ is the Frobenius endomorphism, and the modules $M$ and $N$ are assumed to be
isomorphic to $M \otimes_R \sr R $ and $\Hom_R(\sr R,N)$, respectively.

\begin{lemma}
\label{freeinj}
 Let $\varphi \colon R \to R$ be a finite contracting endomorphism. Let $M$ be a noetherian $R$-module and $N$  an artinian $R$-module. The following statements hold:
\begin{enumerate}[\quad\rm(a)]
\item $M$ is free if and only if $M$ is a direct summand of $M \otimes_R \sr R $;
\item $N$ is injective if and only if $N$ is a direct summand of $\Hom_R(\sr R,N)$.
\end{enumerate}
\end{lemma}

\begin{proof}
(a) The ``only if''  part is obvious. For the converse, let $F_1 \xra {\theta} F_0 \to M \to 0$ be a minimal free presentation. It induces a minimal   free presentation
\[
F_1 \otimes_R \sr R \xra {\theta \otimes_R \sr R}  F_0 \otimes_R \sr R \to M \otimes_R \sr R \to 0
\]
Thus, $M$ and $M \otimes_R \sr R$ have the same minimal number of
generators, and this implies  $M \cong M \otimes_R \sr R$, since $M$
is a direct summand of $M \otimes_R \sr R$.

Let $I(\theta)$ be the ideal generated by the entries of a matrix
representing $\theta$. It is contained in $\mf m$, and one has $I (\theta)=I(\theta \otimes_R
\sr R)=\varphi(I(\theta))$. Let $i$ be a positive integer such that
$\varphi^i(\mf m) \subseteq \mf m^2$. Thus, for every positive integer $n$
one has $I(\theta)= \varphi^{in}(I(\theta)) \subseteq \mf m^{2n}$. Krull's
Intersection Theorem now implies $I(\theta) = 0$, so $M$ is free.

(b)  Every artinian injective module is a finite direct sum of  the
injective envelope of the residue field, and the finiteness of $\varphi$
implies
\[
\lbrace \mf q \in \Spec R \ \vert \ \varphi^{-1}(\mf q) =
\mf m \rbrace =\lbrace \mf m \rbrace\,.
\]
Thus, the  ``only if''  part follows from Theorem~\eqref{injbasechange}.

Conversely, suppose $\Hom_R(\sr R,N) \cong N \oplus N'$ for some $R$-module
$N'$.  Assume first that $R$ is complete and set $(-)^\vee = \Hom_R(-,E_R(R/{\mf m}))$.
One then has
\[
N^\vee \oplus ({N'})^\vee \cong (N \oplus N')^{\vee} \cong \Hom_R(\sr R,N)^{\vee} 
\cong \sr R \otimes_RN^{\vee}
\]
where the last  isomorphism holds by \cite[(3.60)]{R}. Thus, part (a) shows that
$N^\vee$ is  free. This implies that $ N^{\vee \vee}$ is injective. The
natural map $N \to N^{\vee\vee}$  of $R$-modules is an isomorphism,
by Matlis Duality~\cite[(3.2.12)]{BH}, so $N$ is injective.

For a general local ring $R$, let $\ \widehat{}\ $  denote the functor
of $\mf m$-adic completion.  The finiteness of $\varphi$ implies that
the $R$-module $\Hom_R( \sr R,N)$ is artinian along with $N$.  Thus,
both modules have natural structures of $\widehat R$-modules, and the
second one is a direct summand of the first.  For it we have isomorphisms
\begin{align*}
\Hom_R( \sr R,  N)
&\cong\Hom_R(\sr R, \Hom_{\widehat R}(\widehat R, N))\\
&\cong\Hom_{\widehat R}(\sr  R \otimes_R \widehat R,  N)\\
& \cong \Hom_{\widehat R}(\srh \widehat R,  N)\,;
\end{align*}
the last one is induced by $ \srh \widehat R \cong \sr  R \otimes_R \widehat R $, 
which is due to the finiteness of $\varphi$.

{}From the already settled case of complete rings we conclude that $N$ is an
injective $\wh R$-module.  Adjointness yields an isomorphism $\Hom_R(-,  N)\cong
\Hom_{\widehat R}( \widehat R \otimes_R - ,  N)$, which shows that it is
injective over $N$, as well.
 \end{proof}

Let $\gam(M)$ denote the $\mf m$-torsion submodule,
$\cup_{i\geqslant 0}(0:_{M}{\mf m}^{i})$, of an $R$-module $M$.
 
\begin{lemma}
\label{localcohomology}
Let $\psi \colon (R,\mf m) \to (S,\mf n) $ be a finite local homomorphism of local rings.

 If  $N$ is a finite $S$-module, then for every $R$-module $M$ one has an isomorphism of $S$-modules
\[
\Hom_R(N,\gam(M)) \cong \gam[\mf n](\Hom_R( N,M))
\]
\end{lemma}

\begin{proof}
Let 
$\beta \colon \gam(M) \to M$ be the inclusion map. Consider the  maps
\begin{equation*}
\xymatrix{
\gam(\Hom_R( N,M)) \ar[r]^{\alpha} & \Hom_R( N,M)  & & \Hom_R(N,\gam(M)) \ar[ll]_{\Hom_R(N,\beta)} }
\end{equation*}
where $\alpha$ is the inclusion.  Let  $U$ and $V$ be the images of
$\alpha$ and $\Hom_R(N,\beta)$, respectively. These maps are injective,
so  it is enough to show  $U=V$.  Let $\beta \circ f $ be an element
of $V$. Since $M$ is finite over $R$, there  exists a positive integer
$m$ such that $\mf m^m(\beta \circ f)(x)=0$ for each $x \in N$.  Thus,
one has $\psi(\mf m^m)(\beta \circ f)(x)=0$  for each $x \in  N$. Since
the length of $S/{\mf mS}$ is finite, there exists a positive integer
$n$ such that $\mf n^n \subseteq \mf mS$.  Therefore, one has $\mf
n^{mn}(\beta \circ f)=0$, hence $\beta \circ f $ is in $U$.

Now let $g$ be an element of $U$. There exits an integer $n$ such that
$\mf n^n g = 0$, because $N$ is finite over $S$.  Since $\psi$ is local,
one has $\psi(\mf m^n)\subseteq \mf n^n$, hence $\mf m^n g(x)= \psi(\mf
m^n) g(x) =0$, for all $x\in  N$ and this implies that $g$ is in $V$.
\end{proof}

We also need the following special case of \cite[(2.5.8)]{EGA}:

\begin{chunk}
\label{completion}
If $M$ and $N$ are finite $R$-modules, then $\widehat M \cong \widehat N$ as $\widehat R$-modules implies
$M \cong N$ as $R$-modules.
\end{chunk}

The next theorem is Theorem~\ref{theorem1} from the introduction.

\begin{theorem}
\label{maintheorem}
Let $\varphi \colon R \to R$ be a finite contracting endomorphism and $M$ a finite
$R$-module. The following are equivalent:
\begin{enumerate}[\quad \rm (a)]
\item  $R$ is Cohen-Macaulay and $M$ is Gorenstein.
\item $M\cong \Hom_R(\sr R,M)$  and one has $\Ext^i_R(\sr R,M) = 0$ for $i \geq 1$.
\item $M$ is a direct summand of $\Hom_R(\sr R,M)$ as an $R$-module, and one has
  $\Ext^i_R(\sr R,M) = 0$ for $1 \leq i \leq \depth M$.
\end{enumerate}
\end{theorem}

As a corollary, we characterize Gorenstein rings among rings admitting finite contracting
endomorphisms:

\begin{corollary}
  \label{maincorollary}
The ring $R$ is Gorenstein if and only if one has $\Hom_R(\sr R,R)\cong  R$ 
and $\Ext^i_R(\sr R,R) = 0$ for all  $1\leq i \leq \depth R$.
  \qed
\end{corollary}

This result above was proved by Goto~\cite[(1.1)]{G} when $\varphi$ is the Frobenius
endomorphism of a local ring of positive characteristic.  The plan of our proof is similar
to that of Goto's, but new ideas are required to implement it. One difficulty is that a
contracting endomorphism need not induce a bijection on $\operatorname{Spec}(R)$.

\begin{Remark}
The implication (c)$\implies$(a) may fail if $\varphi$ is not
contracting. For example, when $\varphi$ is the identity map, (c) holds
for every ring $R$.

Also, in part (c)  the  conditions on the  $\Hom$ and $\Ext$ modules are
independent. Indeed, if $M$ is a module of depth zero, then obviously
$M$ satisfies the condition on $\Ext$ modules, but it need not  be
Gorenstein. Moreover, Goto in \cite{G} provides an example showing
that $\Hom_R(\sr R,R) \cong R$ does not imply that $R$ is Gorenstein.
\end{Remark}

\begin{mainthmproof}
(a) $\implies$ (b): 
One  has an isomorphism of $\widehat R$-modules
\[
\widehat R  \otimes_R \Ext^i_R(\sr R,M) \cong \Ext^i_{\widehat R} (\srh \widehat R,\widehat M)
\]
Thus, $\widehat M\cong \Hom_{\widehat R} (\srh \widehat R,\widehat M)$ holds  if and only if  $M\cong \Hom_R(\sr R,M)$ holds; see \eqref{completion}. Moreover, one has $\ext i {\widehat R} {\srh \widehat R} {\widehat M}=0$ if and only if $\ext i R {\sr R} M =0$. So, we may assume that $R$ is complete, and hence that it admits a canonical module, say, $\omega$.

Since $M$ is a Gorenstein module, one has $ M \cong \omega^n$ for some positive integer $n$; see  \cite[(2.7)]{S2}. The $R$-module $\sr  R $ is maximal Cohen-Macaulay so \cite[(3.3.3)]{BH} yields:
\[
\ext i R {\sr R} M \cong  \ext i { R}  {\sr R} { \omega } ^n =0 \quad  \text{for all}  \quad i \ge 1
\]
Moreover, one has 
\[ 
\Hom_R(\sr R, M) \cong \Hom_{ R}(\sr R,\omega)^n
\]
The $R$-module $\Hom_{ R}(\sr R,\omega)$ is a canonical module for $ R$, see \cite[(3.3.7)]{BH}. Since
canonical modules are unique up to isomorphism, by \cite[(3.3.4)]{BH}, one has an
isomorphism of $ R$-modules $\Hom_{ R}(\sr R,\omega) \cong \omega$. Thus, one
obtains 
\[
\Hom_R(\sr R,M) \cong \omega^n \cong M\,.
\]

(b) $\implies$ (c) is obvious.

(c) $\implies$ (a): Set $g = \depth M$. It is enough to show that $\lch{g}(M)$, the $g$th
local cohomology of $M$ with respect to $\mf m$, is injective; see \cite[(3.9)]{FFGR}.
Fix a minimal injective resolution of $M$ over $R$, say
\[
I^*=\quad
0 \arr I^0 \xra {d^0} I^1 \arr \cdots \arr I^g \xra {d^g} I^{g+1} \arr \cdots
\]
Since $\Ext^i_R(\sr R,M) = 0$ for all $1 \leq i \leq s$, the induced sequence of
$R$-modules
\begin{align*}
\label{injres}
0 \arr \Hom_R(\sr R,I^0) \xra {d^0_{\ast}} \Hom_R(\sr R,I^1) \arr \cdots \arr \Hom_R(\sr R,I^g) \xra
{d^g_{\ast}} \Hom_R(\sr R,I^{g+1})
\end{align*}
has non-zero homology only in degree $0$, where it is equal to $\Hom_R(\sr R,M)$.  As each
$\Hom_R(\sr R,I^j)$ is injective, it can be extended to an injective resolution, say
$J^*$, of $\Hom_R(\sr R,M)$ over $R$.  By hypothesis, $M$ is a direct summand of
$\Hom_R(\sr R,M)$, as $R$-modules, so the complex $I^*$ is a direct summand of the complex
$J^*$.

Applying the functor $\gam(-)$ and using Lemma~\ref{localcohomology}, one sees 
that the complex $\gam(I^*)$ is a direct summand of the complex 
$\Hom_R(\sr R,\gam(J^*))$.  One has $\gam(I^i) = 0$ for $0 \leq i < g$, see \cite [(3.2.9)]{BH}, so one gets a commutative diagram of $R$-linear maps
\begin{equation*}
\xymatrixcolsep{1.2pc}
\xymatrix{
0 \ar[r]
&\lch g(M) \ar[r]
&\gam(I^g) 
\ar[r]^{\gam(d^g)}\ar[d]^{\gam(f^g)}
&\gam(I^{g+1}) 
\ar[d]^{\gam(f^{g+1})}
\\
0 \ar[r]
 & \Ker \Hom_R(\sr R,\gam(d^g)) \ar[r]
 &\Hom_R(\sr R,\gam(I^g)) \ar[r]
 & \Hom_R(\sr R,\gam(I^{g+1}))
}
\end{equation*}
where the vertical ones are split-injective.  It shows 
that $\lch g(M)$ is a direct summand of $ \Ker \Hom_R(\sr R,\gam(d^g)) $, and that 
this latter is isomorphic to $ \Hom_R(\sr R,\Homology^g_{\mf m}(M))$.
Therefore, Lemma~\eqref{freeinj} implies that the $R$-module $\Homology^g_{\mf m}(M)$ is
injective.
\end{mainthmproof}


\section{Gorenstein Dimension}

Iyengar and Sather-Wagstaff \cite{IS} show that a local ring $R$ equipped with a
contracting endomorphism is Gorenstein if and only if the Gorenstein dimension of $\sr R$
is finite. Here we give a short proof of this result when $\varphi$ is finite.

\begin{chunk}
  Let $M$ be a finite $R$-module. One says that $M$ is \emph{totally reflexive} if
\begin{enumerate}[\quad \rm (a)]
\item The canonical map $M \to \Hom_R(\Hom_R(M,R),R)$ is an isomorphism;
\item $\Ext^i_R (M,R) = 0$ for all $i>0$; and
\item $\Ext^i_R(\Hom_R(M,R),R)=0$ for all $i>0$.
\end{enumerate}
The \emph{Gorenstein dimension} of $M$, or $\operatorname{G-dim}_R(M)$, is 
defined to be
the least integer $n$ such that the $n$th syzygy of $M$ is totally reflexive.  
\end{chunk}

For every finite $R$-module $N$, set $\type N=\rank_k\Ext^g_R(k,N)$,
where $g=\depth N$.

\begin{chunk}
\label{semidualizing}
A finite $R$-module $C$ is \emph{semi-dualizing} if the canonical map 
$R \to \Hom_R(C,C) $ is an isomorphism and $\ext i R C C =0$ for all 
$i \geq 1$. For example, the ring $R$ is a semi-dualizing module. Each  semi-dualizing module $C$ satisfies an equality
\begin{align}
 \label{type}
 \type R = \mu_R  (C) \type C
\end{align}
 where $\mu_R(C)$ denotes the minimal number of generators of $C$; \cite[(3.18.2)]{C}.
\end{chunk}

The next result contains Theorem~\ref{theorem3} from the introduction.

\begin{theorem}
\label{gdim}
Let  $\varphi \colon R \to R$ be a finite  contracting homomorphism.
The ring $R$ is Gorenstein if and only if $\operatorname{G-dim}_R(\sr R)$ is finite,
if and only if the $R$-module $\sr R$ is totally reflexive.
\end{theorem}

\begin{proof}
Set $g=\depth R$.  It is convenient to set $S=R$, and to view $\varphi$ as a local homomorphism
$\varphi\colon R\to S$ and $S$ as a left $R$-module via $\varphi$. Note that $S \cong \sr R$ as left $R$-modules.
Recall that when $X$ is a complex of $R$-modules $\Rhom R X -$  denotes
the right derived functor of $\Hom_R(X,-)$, and  $\dtensor X R -$ denote the left 
derived functor of  $X \otimes _R -$.

When $R$ is Gorenstein each finite $R$-module, and hence $S$, has finite $G$-dimension.

Assume $ \operatorname{G-dim}_R(S)$ is finite. As one has $\depth_R S = g$, 
so the Auslander-Bridger formula, see \cite[(4.13)]{AB}, shows that 
the $R$-module $S$ is totally reflexive. 

Finally, assume that $S$ is totally  reflexive over $R$, and set $C=\Hom_R(S,R)$.  
One then has $\Ext^{i}(S,R)=0$ for $i\geq 1$, so $C$ is  quasi-isomorphic to 
$\Rhom R SR$
as complexes of $S$-modules.  This gives the first isomorphism of $S$-modules:
\begin{align*}
\Rhom SCC & = \Rhom S{\Rhom R SR} {\Rhom R SR} \\ 
          & \cong \Rhom R{\Rhom RSR} R \\ 
          & \cong S
\end{align*}
The second one holds by adjuntion, and the last one by total reflexivity.
Thus, the $S$-module $C$ is semi-dualizing. One thus has the following isomorphisms:
\[
\Rhom SkC = \Rhom Sk {\Rhom R SR} \cong \Rhom R k R\,.
\]
Since the rings $S$ and $R$ are equal, they imply $\type C = \type S$.  Formula \eqref{type} shows
that the $S$-module $C$ is  cyclic. Moreover, $\Ann_{S}C=\{0\}$, since the  canonical map $S \to \Hom_S(C,C)$ is bijective.  We thus conclude that $C$ is isomorphic
to $\Hom_R(S,R)$, so $S$, and hence $R$, is Gorenstein by Corollary
\ref{maincorollary}.
\end{proof}

\section*{Acknowledgments}
I am grateful to my advisers, Luchezar  Avramov and Srikanth Iyengar, for their guidance.
 I also thank Ryo Takahashi for  helpful conversations.

\end{document}